\documentclass[leqno,12pt]{article} 
\usepackage{amsmath, amssymb}
\usepackage{amsthm} 

\usepackage{graphicx}
\usepackage{float}

\theoremstyle{plain} 
\newtheorem{theorem}{\indent\sc Theorem}[section]
\newtheorem{lemma}[theorem]{\indent\sc Lemma}
\newtheorem{corollary}[theorem]{\indent\sc Corollary}
\newtheorem{proposition}[theorem]{\indent\sc Proposition}

\theoremstyle{definition} 
\newtheorem{definition}[theorem]{\indent\sc Definition}
\newtheorem{remark}[theorem]{\indent\sc Remark}

\makeatletter
\def\address#1#2{\begingroup
\noindent\parbox[t]{7.8cm}{%
\small{\scshape\ignorespaces#1}\par\vskip1ex
\noindent\small{\itshape E-mail address}%
\/: #2\par\vskip4ex}\hfill%
\endgroup}%
\makeatother
\title{\uppercase{Ollivier Ricci curvature on graphs obtained by removing edges from complete graphs}} 
\author{
\bigskip \\
\textsc{Yui Asai$^{*}$ and Taiki Yamada} 
}
\date{} 
\begin{document}

\maketitle

\footnote{ 
$^{*}$Corresponding author.
}
\begin{abstract}
Under what conditions does the sign of the Ollivier–Ricci curvature on a graph of a certain order change? In this paper, we discuss the curvature of graphs obtained by removing edges from complete graphs, as complete graphs have a stable positive curvature.
We defined graphs obtained by removing matching edges, the set of edges incident with the vertex, and cycle edges from  complete graphs, and then analyzed the Ollivier–Ricci curvature of those graphs.
The results show that the curvature of the graphs in the above three patterns is equal to the value obtained by dividing the number of triangles, including two vertices, by the maximum degree of the two vertices.  This result also indicates that the curvature of the above graphs is zero or positive.
This study concludes that the Ollivier–Ricci curvature is predicted to be positive even if some edges are removed from a complete graph, and we suggest that these discussions are suitable for investigating the conditions under which the sign of the Ollivier–Ricci curvature on a graph.

\end{abstract}

\section{Introduction}\label{Sec:Introduction}
The notion of curvature has emerged as a powerful organizing principle for analyzing complex networks~\cite{Lin-Lu-Yau,Ollivier2009,Sandhu2015}. Among several discretizations of Ricci curvature on graphs, the Ollivier--Ricci curvature has attracted particular attention owing to its optimal-transport foundation and ability to capture coarse geometric properties, such as robustness, expansion, and information flow~\cite{Jost-Liu, Ollivier2009,Ollivier2010}. Since its introduction by Ollivier~\cite{Ollivier2007,Ollivier2009}, this curvature has been widely applied in network science, data analysis, and discrete geometry, where it serves as a proxy for structural cohesion and resilience. ~\cite{NiLubichGao2019,Sandhu2015,Sia2019}.

A fundamental and well-understood benchmark in this theory is a complete graph. Complete graphs exhibit a strictly positive Ollivier--Ricci curvature on every edge under standard random-walk measures~\cite{Jost-Liu,Lin-Lu-Yau}. This positivity reflects their extreme connectivity: every pair of vertices shares all possible common neighbors, yielding the maximal local overlap in the probability measures used in the Wasserstein transport defining the curvature~\cite{Ollivier2009,Villani2009}. In this sense, complete graphs represent the ``maximally curved'' discrete spaces in the category of simple graphs.

At the opposite end of the spectrum, negatively curved graphs are associated with tree-like structures, bottlenecks, and sparse connectivity~\cite{BhattacharyaMukherjee2015,Jost-Liu,NiLubichLin2015}. Numerous studies have established qualitative relationships between the curvature sign and graph-theoretic properties such as expansion~\cite{KlartagKozmaRalli2016}, clustering~\cite{Sia2019}, and robustness under perturbations~\cite{NiLubichGao2019,Sandhu2015}. However, the precise structural transition from positive to non-positive curvature remains poorly understood. In particular, it is not known in a systematic way how much connectivity must be removed from a maximally connected graph and at which locations before the local geometry, as measured by the Ollivier--Ricci curvature, loses its positivity.

This gap is especially striking because edge deletion is the most elementary graph operation and models a wide range of realistic processes, such as link failures in infrastructure networks, pruning in communication graphs, sparsification for computational efficiency~\cite{SpielmanTeng2011}, and evolutionary degradation of connectivity~\cite{NiLubichGao2019}. Therefore, from both theoretical and applied perspectives, the following question arises:

\begin{quote}
\textit{Which edges, and how many, must be removed from a complete graph so that some edges attain non-positive Ollivier--Ricci curvature?}
\end{quote}

This question lies at the intersection of discrete differential geometry and extremal graph theory~\cite{Bollobas2004}. It seeks a threshold phenomenon: a characterization of minimal edge removals that induce a sign change in the curvature. Unlike global graph invariants, the Ollivier--Ricci curvature is inherently local and depends delicately on the combinatorial configuration of common neighbors around each edge~\cite{Jost-Liu,Lin-Lu-Yau}. Consequently, both the number of removed edges and their placement play a decisive role.

Understanding this transition has several important implications for the field. First, it provides a geometric interpretation of graph sparsification by identifying how far a dense graph can be thinned while preserving positive curvature, hence maintaining strong local cohesion~\cite{Sandhu2015,SpielmanTeng2011}. Second, it offers a new viewpoint on network vulnerability, as edges with non-positive curvature often correspond to structural bottlenecks~\cite{NiLubichGao2019,Sia2019}. Third, it contributes to the broader program of relating discrete curvature to classical extremal problems, connecting curvature theory with questions traditionally studied in combinatorics~\cite{Bollobas2004,Lin-Lu-Yau}.

In this work, we initiate a systematic study of edge deletions from complete graphs from the viewpoint of the Ollivier--Ricci curvature. We analyze how the curvature of an edge depends on the combinatorial pattern of missing adjacent edges, and we derive the conditions under which the curvature becomes zero or negative. Our results reveal that the transition from positive to non-positive curvature is governed not only by the total number of deleted edges but also by specific local configurations that disrupt common-neighbor structures.

By characterizing these configurations, we provide the first structural description of when a complete graph, after edge removal, ceases to be positively curved in the sense of Ollivier. This establishes a concrete bridge between discrete curvature and extremal graph modifications and opens a pathway toward a curvature-based theory of graph robustness and sparsification.

The remainder of this paper is organized as follows.
In Section \ref{Preliminaries}, we introduce the necessary concepts of graph theory and define some edge sets contained in complete graphs.  In Section \ref{Ricci Curvature}, we describe the Ollivier--Ricci curvature and both upper and lower bounds for the curvature on graphs. In Section \ref{Main Results}, we show the curvature of some graphs obtained by removing matching edges, adjacent edges to a vertex, and cycle edges from complete graphs. In Section \ref{Conclusion}, we provide a summary of the curvature on above graphs.

\section{Preliminaries}
\label{Preliminaries}
In this paper, we assume that $G=(V,E)$ is an undirected graph, where $V$ is the set of the vertices and $E$ the set of edges. For any subset $F \subset E$, $G - F$ denotes the graph obtained by removing the edges contained in $F$.

\begin{definition}
For any two vertices $s$ and $t$, a \textit{path} between $s$ and $t$ represents a sequence of edges $\left\{ (v_i, v_{i+1}) \right\}_{i=0}^{n-1}$, where $v_0 = s,\ v_n = t$.
\end{definition}

\begin{remark}
If there exists a path between $s$ and $t$, then $s$ and $t$ is called \textit{connected}. In particular, if any pair of vertices in $V$ is connected, then $G=(V, E)$ is called \textit{connected}. 
\end{remark}

\begin{definition}
The {\em distance} $d(u, v)$ between two vertices $u, v \in V$ is given by the length of a shortest path between $u$ and $v$.    
\end{definition}

\begin{definition}
For any vertex $u \in V$, the \textit{neighborhood} of $u$ is defined by
\begin{eqnarray*}
\Gamma (u) = \left\{ v \in V \mid (u, v) \in E \right\}.
\end{eqnarray*}
\end{definition}

\begin{definition}
For any vertex $u \in V$, the \textit{degree} of $u$, denoted by $d_u$, is the number of edges incident to $u$.
\end{definition}

We assume the following conditions:
\begin{enumerate}
\renewcommand{\labelenumi}{(\arabic{enumi})}
\item Local finiteness 
\item Simpleness (i.e., there exist no loops and no multi-edges)
\item Connectedness 
\end{enumerate}

\begin{definition}
    Let $G$ be a graph. A $t$-\textit{matching} of $G$ is a set of $t$ edges that do not share any common vertex.
    We denote $\lbrace te\rbrace$ as the edge set of $t-matching$.\par 
    A perfect matching $M$ of a graph $G$ is a set of non-adjacent edges such that every vertex of $G$ is incident to exactly one edge of the matching $M$.
\end{definition}

\begin{definition}
    Let $G$ be a graph. Choosing a vertex of $G$ and $t$ adjacent edges to the vertex, we denote the set of edges as $[te]$.  
\end{definition}

\section{Ricci Curvature}
\label{Ricci Curvature}
\begin{definition}
The 1-Wasserstein distance between any two probability measures $\mu$ and $\nu$ on $V$ is given by
	\begin{eqnarray*}
	W(\mu, \nu) = \inf_{A} \sum_{u, v \in V}A(u, v)d(u, v),
	\end{eqnarray*}
	where $A : V \times V \to [0, 1]$ runs over all maps satisfying 
	\begin{eqnarray}
	\label{coupling}
		\begin{cases}
		\sum_{v \in V}A(u, v) = \mu(u),\\
		\sum_{u \in V}A(u, v) = \nu(v).
		\end{cases}
	\end{eqnarray}
Such a map $A$ is called a {\em coupling} between $\mu$ and $\nu$. 
\end{definition}

\begin{proposition}[Kantorovich, Rubinstein]
\label{kantoro}
The $1$-Wasserstein distance between any two probability measures $\mu$ and $\nu$ on $V$ is written as
		\begin{eqnarray*}
		W(\mu, \nu) = \sup_{f} \sum_{u \in V} f(u)(\mu(u) - \nu(u)),
		\end{eqnarray*}
where the supremum is taken over all functions $f$ on $V$ that satisfy $|f(u) - f(v)| \leq d(u, v)$ for any $u, v \in V$.
	\end{proposition}
A function $f$ on $V$ is said to be {\em $1$-Lipschitz} if $|f(u)-f(v)| \leq d(u,v)$ for any $u, v \in V$.

\begin{definition}
For any vertex $x \in V$, we define a probability measure $m_x$ on $V$ by
	\begin{eqnarray*}
    	m_x(v) = 
    		\begin{cases}
    		\cfrac{1}{d_{x}}, &\mathrm{if}\ (x, v) \in E, \\
    		0, & \mathrm{\mathrm{otherwise}}.
    		\end{cases}
   	\end{eqnarray*}
\end{definition}

\begin{definition}
	\label{Ricci}
    For any vertices $x, y \in V$, let $\mu_x$ and $\nu_y$ be probability measures on $V$.
    Then, the {\em Ollivier Ricci curvature} of $x$ and $y$ is defined as
   	\begin{eqnarray*}
    	\kappa(x, y) =  1 - \cfrac{W(\mu_x, \nu_y)}{d(x, y)}.
   	\end{eqnarray*}
\end{definition}

\begin{theorem}[\cite{Jost-Liu}]
\label{リッチ曲率の下限}
For any pair of neighboring vertices $x, y \in V$, the number of triangles including $x$ and $y$ as vertices is denoted by $\# (x, y)$. Then we have
\begin{eqnarray*}
\kappa (x, y) \geq - \left( 1- \cfrac{1}{d_x} - \cfrac{1}{d_y} - \cfrac{\# (x, y)}{d_x \wedge d_y} \right)_+ - \left( 1- \cfrac{1}{d_x} - \cfrac{1}{d_y} - \cfrac{\# (x, y)}{d_x \vee d_y} \right)_+ + \cfrac{\# (x, y)}{d_x \vee d_y},
\end{eqnarray*}
where $s_+ := \max (s, 0)$, $s \vee t := \max (s, t)$, and $s \wedge t := \min (s, t)$ for real numbers $s$ and $t$.
\end{theorem}
Note that any tree attains the lower bound. This coincides with the geometric intuition of curvature. On the other hand, Jost and Liu provided the following upper bound for the Ricci curvature.
\begin{theorem}[\cite{Jost-Liu}]
\label{リッチ曲率の上限}
For any pair of neighboring vertices $x, y \in V$, we have
\begin{eqnarray*}
\kappa(x, y) \leq \cfrac{ \# (x, y)}{d_{x} \vee d_{y}}.
\end{eqnarray*}
\end{theorem}

\section{Main Results}
\label{Main Results}
The main result of this paper is the determination of how the Ollivier-Ricci curvature changes when edges are removed from a complete graph using three different methods. In this section, we discuss each of these methods. 
Since the labeling scheme assigned to the vertices of the graph differs for each method, we explain the labeling scheme at the beginning of each subsection.

\subsection{Case where matching edges are removed}
In this subsection, we consider a graph $K_{n}-\lbrace me\rbrace$ obtained by removing matching edges from a complete graph. First, we relabeled the vertices according to their degree (see Figure \ref{マッチングエッジの例}). Let $v_{i}\ (i=1,\ldots,n-2m)$ denote the vertex with degree $n-1$, and $u_{i}\ (i=1,\ldots,2m)$ denote the vertex with degree $n-2$ on $K_{n}-\lbrace me\rbrace$. 
\begin{figure}[H]
\centering
\includegraphics[height=6cm]{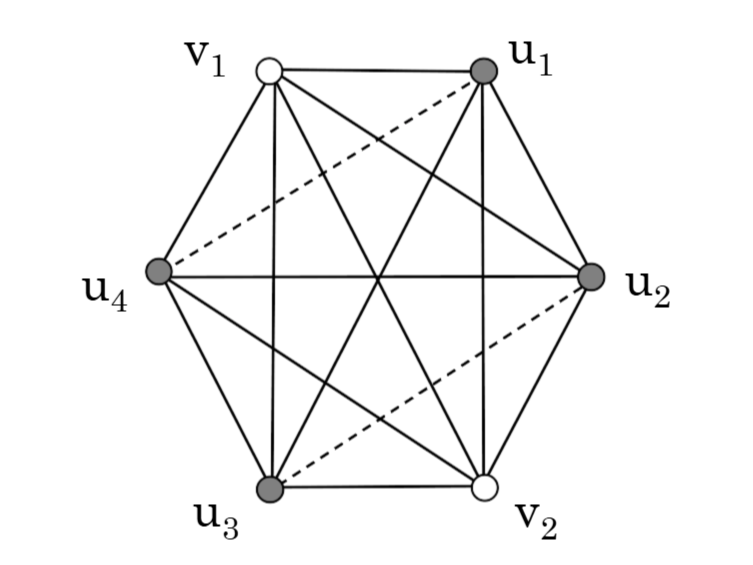}
\caption{$K_{6}-\lbrace2e\rbrace$}
\label{マッチングエッジの例}
\end{figure}

We provide the following lemma concerning the number of triangles in the case of a graph excluding matching edges to calculate the Ollivier–Ricci curvature.

\begin{lemma}
\label{マッチングの三角形数}
    We consider a graph $K_{n}-\lbrace me\rbrace$$(1\leq m\leq\lfloor n/2\rfloor)$.
   Then for any pair of neighboring two vertices,
    \begin{gather*}
        \#(v_{i},v_{j})=n-2, \\
        \#(v_{i},u_{j})=n-3,\\
        \#(u_{i},u_{j})=n-4.
    \end{gather*}
\end{lemma}
\begin{proof}
    First, we consider any pair of neighboring vertices $v_{i},v_{j}$. We know that vertex $v_{j}$ has $n-2$ neighbors, except for $v_{i}$, and $v_{i}$ is adjacent to all of them. Therefore,
    \begin{eqnarray*}
        \#(v_{i},v_{j})=n-2.
    \end{eqnarray*}
    
    Second, we consider any pair of neighboring vertices $v_{i},u_{j}$. We know that $u_{j}$ has $n-3$ neighbors, except for $v_{i}$, and $v_{i}$ is adjacent to all of them. Therefore,
    \begin{eqnarray*}
        \#(v_{i},u_{j})=n-3.
    \end{eqnarray*}
    
    Finally, we consider any pair of neighboring vertices $u_{i},u_{j}$. We know that $u_{j}$ has $n-3$ neighbors, except for $u_{i}$. $u_{i}$ is not adjacent to one of the neighbors of $u_{j}$, but is adjacent to all others. Therefore,
    \begin{eqnarray*}
        \#(u_{i},u_{j})=(n-3)-1=n-4.
    \end{eqnarray*}
\end{proof}

By applying Lemma\ref{マッチングの三角形数}, the curvature is calculated as follows.

\begin{proposition}
    We consider a graph $K_{n}-\lbrace me\rbrace$$(1\leq m\leq\lfloor n/2\rfloor)$. 
    Then for any pair of neighboring two vertices $x,y$ on $G$, we have
    \begin{eqnarray*}
        \kappa(x,y)=\cfrac{\#(x,y)}{d_{x}\vee d_{y}}=
  \begin{cases}
    \cfrac{n-2}{n-1}, & \text{if $x=v_{i},y=v_{j}$}, \\
    \cfrac{n-3}{n-1},       & \text{if $x=v_{i}$, $y=u_{j}$}, \\
    \cfrac{n-4}{n-2} ,      & \text{if $x=u_{i}$, $y=u_{j}$}.
  \end{cases}
    \end{eqnarray*}
\end{proposition}

\begin{proof}
    First, we consider any pair of neighboring vertices $v_{i},v_{j}$.By assumption and Lemma\ref{マッチングの三角形数}, we know
    \begin{eqnarray*}
        d_{v_{i}}\vee d_{v_{j}}=d_{v_{i}}\wedge d_{v_{j}}=n-1,\ \#(v_{i},v_{j})=n-2.
    \end{eqnarray*}
    Therefore, by Theorem \ref{リッチ曲率の下限} and Theorem \ref{リッチ曲率の上限},
    \begin{eqnarray*}
        \cfrac{n-2}{n-1}\geq
        \kappa(x,y)\geq -2\left( 1- \cfrac{1}{n-1} - \cfrac{1}{n-1} - \cfrac{n-2}{n-1} \right)_++ \cfrac{n-2}{n-1}=\cfrac{n-2}{n-1},
    \end{eqnarray*}
    which implies
    \begin{eqnarray*}
    \kappa(v_{i},v_{j})=\cfrac{n-2}{n-1}=\cfrac{\#(v_{i},v_{j})}{d_{v_{i}}\vee d_{v_{j}}}.
    \end{eqnarray*}

    Second, we consider any pair of neighboring vertices $v_{i},u_{j}$.By assumption and Lemma\ref{マッチングの三角形数}, we know
    \begin{eqnarray*}
        d_{v_{i}}\vee d_{u_{j}}=n-1,\ d_{v_{i}}\wedge d_{u_{j}}=n-2,\ \#(v_{i},u_{j})=n-3.
    \end{eqnarray*}
    Therefore, by Theorem \ref{リッチ曲率の下限} and Theorem \ref{リッチ曲率の上限},
    \begin{eqnarray*}
        \cfrac{n-3}{n-1}\geq\kappa(v_{i},u_{j})&\geq& -\left( 1- \cfrac{1}{n-1} - \cfrac{1}{n-2} - \cfrac{n-3}{n-2} \right)_+\\ &\ & - \left( 1- \cfrac{1}{n-1} - \cfrac{1}{n-2} - \cfrac{n-3}{n-1} \right)_+ + \cfrac{n-3}{n-1}=\cfrac{n-3}{n-1},
    \end{eqnarray*}
    which implies
    \begin{eqnarray*}
    \kappa(v_{i},u_{j})=\cfrac{n-3}{n-1}=\cfrac{\#(v_{i},u_{j})}{d_{v_{i}}\vee d_{u_{j}}}.
    \end{eqnarray*}

    Finally, we consider any pair of neighboring vertices $u_{i},u_{j}$.By assumption and Lemma\ref{マッチングの三角形数}, we know
    \begin{eqnarray*}
        d_{u_{i}}\vee d_{u_{j}}=d_{u_{i}}\wedge d_{u_{j}}=n-2, \#(u_{i},u_{j})=n-4.
    \end{eqnarray*}
    Therefore, by Theorem \ref{リッチ曲率の下限} and Theorem \ref{リッチ曲率の上限},
    \begin{eqnarray*}
        \cfrac{n-4}{n-2}\geq\kappa(u_{i},u_{j})\geq -2\left( 1- \cfrac{1}{n-2} - \cfrac{1}{n-2} - \cfrac{n-4}{n-2} \right)_+ + \cfrac{n-4}{n-2}
        = \cfrac{n-4}{n-2},
    \end{eqnarray*}
    which implies
    \begin{eqnarray*}
    \kappa(u_{i},u_{j})=\cfrac{n-4}{n-2}=\cfrac{\#(u_{i},u_{j})}{d_{u_{i}}\vee d_{u_{j}}}.
    \end{eqnarray*}
\end{proof}

\begin{corollary}
          We consider a graph $K_{2m}-M$, where $M$ is a perfect matching in $K_{2m}$.
    Then for any neighboring $x,y$, we have
    \begin{eqnarray*}
    \kappa(x,y)=\dfrac{\#(x,y)}{d_{x}\vee d_{y}}= \dfrac{2m-4}{2m-2}.
    \end{eqnarray*} 
\end{corollary}

\subsection{Case where edges are intensively removed}
In this subsection, we consider the graph $K_{n}-[me]$ obtained by systematically removing edges from a single vertex of a complete graph.
First, we relabeled the vertices according to their degree (see Figure \ref{１点集中の例}). Let $u_{0}$ denote the vertex has degree $n-m-1$, $u_{i}\ (i=1,\ldots,m)$ denote any vertex has degree $n-2$,
    and $v_{i}\ (i=1,\ldots,n-m-1)$ denote any vertex with degree $n-1$ on $K_{n}-[me]$.
\begin{figure}[H]
\centering
\includegraphics[height=6cm]{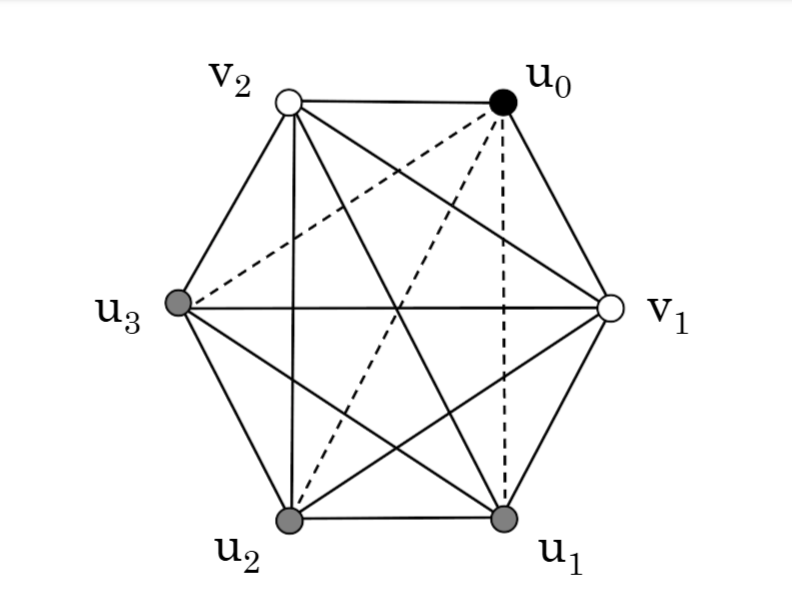}
\caption{$K_{6}-[3e]$}
\label{１点集中の例}
\end{figure}

We provide the following lemma concerning the number of triangles in the case of a graph excluding edges from a single vertex to calculate the Ollivier–Ricci curvature.

\begin{lemma}
\label{1点の辺除去の三角形数}
    We consider a graph $K_{n}-[me]$, where $n\geq4$ and $1\leq m\leq n-3$. 
    Then for any pair of neighboring two vertices, we have
    \begin{eqnarray*}
        \#(v_{i},v_{j})=n-2,\\
        \#(v_{i},u_{j})=n-3,\\
        \#(u_{i},u_{j})=n-3,\\
        \#(v_{i},u_{0})=n-m-2.
    \end{eqnarray*}
\end{lemma}

\begin{proof}
    First, we consider any pair of neighboring vertices $v_{i},v_{j}$. We know that $v_{j}$ has $n-2$ neighbors, except for $v_{i}$, and $v_{i}$ is adjacent to all of them. Therefore,
    \begin{eqnarray*}
        \#(v_{i},v_{j})=n-2.
    \end{eqnarray*}
    
    Second, we consider any pair of neighboring vertices $v_{i},u_{j}$. We know that $u_{j}$ has $n-3$ neighbors, except for $v_{i}$, and $v_{i}$ is adjacent to all of them. Therefore,
    \begin{eqnarray*}
        \#(v_{i},u_{j})=n-3.
    \end{eqnarray*}
    
    Third, we consider any pair of neighboring vertices $u_{i},u_{j}$. We know that vertex $u_{i}$ connects $\lbrace u_{1},\ldots,u_{m}\rbrace\setminus\lbrace u_{i},u_{j}\rbrace$ and $\lbrace v_{1},\ldots,v_{n-m-1}\rbrace$. Similarly, $u_{j}$ connects $\lbrace u_{1},\ldots,u_{m}\rbrace\setminus\lbrace u_{i},u_{j}\rbrace$ and $\lbrace v_{1},\ldots,v_{n-m-1}\rbrace$. Therefore,
    \begin{eqnarray*}
        \#(u_{i},u_{j})=(n-m-1)+(m-2)=n-3.
    \end{eqnarray*}
    
    Finally, we consider any pair of neighboring vertices $v_{i},u_{0}$. We know that $u_{0}$ has $n-m-2$ neighboring except for $v_{i}$, and $v_{i}$ is adjacent to all of them except myself. Therefore, 
    \begin{eqnarray*}
        \#(v_{i},u_{0})=n-m-2.
    \end{eqnarray*}
\end{proof}

By applying Lemma\ref{1点の辺除去の三角形数}, the curvature is calculated as:

\begin{proposition}
    We consider a graph $K_{n}-[me]$, where $n\geq4$ and $1\leq m\leq n-3$.
    Then for any pair of neighboring vertices $x,y$, we have
    \begin{eqnarray*}
        \kappa(x,y)=\cfrac{\#(x,y)}{d_{x}\vee d_{y}}=
        \begin{cases}
    \cfrac{n-2}{n-1}, & \text{if $x=v_{i},y=v_{j}$}, \\
    \cfrac{n-3}{n-1},       & \text{if $x=v_{i}$, $y=u_{j}$}, \\
    \cfrac{n-3}{n-2} ,      & \text{if $x=u_{i}$, $y=u_{j}$},\\
    \cfrac{n-m-2}{n-1} , &\text{if $x=v_{i},y=u_{0}$}.
  \end{cases}
    \end{eqnarray*}
\end{proposition}

\begin{proof}
    First, we consider any pair of neighboring vertices $v_{i},v_{j}$. 
    By assumption and Lemma\ref{1点の辺除去の三角形数}, we know 
    \begin{eqnarray*}
        d_{v_{i}}\vee d_{v_{j}}=d_{v_{i}}\wedge d_{v_{j}}=n-1, \#(v_{i},v_{j})=n-2.
    \end{eqnarray*}
    So, by Theorem \ref{リッチ曲率の下限} and Theorem \ref{リッチ曲率の上限},
    \begin{eqnarray*}
        \cfrac{n-2}{n-1}\geq
         \kappa(v_{i},v_{j})\geq -2\left( 1- \cfrac{1}{n-1} - \cfrac{1}{n-1} - \cfrac{n-2}{n-1} \right)_++ \cfrac{n-2}{n-1}= \cfrac{n-2}{n-1}.
    \end{eqnarray*}
    Thus,
    \begin{eqnarray*}
        \kappa(v_{i},v_{j})=\cfrac{n-2}{n-1}=\cfrac{\#(v_{i},v_{j})}{d_{v_{i}}\vee d_{v_{j}}}.
    \end{eqnarray*}

        Second, we consider any pair of neighboring vertices  $v_{i},u_{j}$. By assumption and Lemma\ref{1点の辺除去の三角形数}, we know
    \begin{eqnarray*}
        d_{v_{i}}\vee d_{u_{j}}=n-1,d_{v_{i}}\wedge d_{u_{j}}=n-2,\#(v_{i},u_{j})=n-3.
    \end{eqnarray*}
    So, by Theorem \ref{リッチ曲率の下限} and Theorem \ref{リッチ曲率の上限},
    \begin{eqnarray*}
       \cfrac{n-3}{n-1}\geq
       \kappa(v_{i},u_{j})&\geq& -\left( 1- \cfrac{1}{n-1} - \cfrac{1}{n-2} - \cfrac{n-3}{n-2} \right)_+\\ 
       &\ &- \left( 1- \cfrac{1}{n-1} - \cfrac{1}{n-2} - \cfrac{n-3}{n-1} \right)_+ + \cfrac{n-3}{n-1}= \cfrac{n-3}{n-1}. 
    \end{eqnarray*}
     Thus, 
    \begin{eqnarray*}
        \kappa(v_{i},u_{j})=\cfrac{n-3}{n-1}=\cfrac{\#(v_{i},u_{j})}{d_{v_{i}\vee}d_{u_{j}}}.
    \end{eqnarray*}

        Third, we consider any pair of neighboring vertices  $u_{i},u_{j}$. By assumption and Lemma\ref{1点の辺除去の三角形数}, we know
    \begin{eqnarray*}
        d_{u_{i}}\vee d_{u_{j}}=d_{u_{i}}\wedge d_{u_{j}}=n-2,\#(u_{i},u_{j})=n-3. 
    \end{eqnarray*}
    So, by Theorem \ref{リッチ曲率の下限} and Theorem \ref{リッチ曲率の上限},
    \begin{eqnarray*}
       \cfrac{n-3}{n-2}\geq
       \kappa(u_{i},u_{j})\geq -2\left( 1- \cfrac{1}{n-2} - \cfrac{1}{n-2} - \cfrac{n-3}{n-2} \right)_++ \cfrac{n-3}{n-2}= \cfrac{n-3}{n-2}. 
    \end{eqnarray*}
    Thus, 
    \begin{eqnarray*}
        \kappa(u_{i},u_{j})=\cfrac{n-3}{n-2}=\cfrac{\#(u_{i},u_{j})}{d_{u_{i}\vee}d_{u_{j}}}.
    \end{eqnarray*}

    Finally, we consider any pair of neighboring vertices $v_{i},u_{0}$. By assumption and Lemma\ref{1点の辺除去の三角形数}, we know
    \begin{eqnarray*}
        d_{v_{i}}\vee d_{u_{0}}=n-1,d_{v_{i}}\wedge d_{u_{0}}=n-m-1,\#(v_{i},u_{0})=n-m-2.
    \end{eqnarray*}
    So, by Theorem \ref{リッチ曲率の下限} and Theorem \ref{リッチ曲率の上限},
        \begin{eqnarray*}
        \cfrac{n-m-2}{n-1}\geq\kappa(v_{i},u_{0})\geq-\left(\cfrac{-m^2+(n-1)m-(n-1)}{(n-1)(n-m-1)}\right)_+ +\cfrac{n-m-2}{n-1}.
    \end{eqnarray*}

        We define
    \begin{eqnarray*}
        B=\cfrac{-m^2+(n-1)m-(n-1)}{(n-1)(n-m-1)}.
    \end{eqnarray*}
    Since we know 
    \begin{eqnarray*}
         \cfrac{1}{(n-1)(n-m-1)}>0
    \end{eqnarray*}
    in $1\leq m\leq n-3$, if $B\leq0$, then we have
    \begin{eqnarray*}
        -m^2+(n-1)m-(n-1)\leq0,
    \end{eqnarray*}
    that is,
    \begin{eqnarray*}
        \left(m-\cfrac{n-1}{2}\right)^2 -\cfrac{n^2-6n+5}{4}\geq0.
    \end{eqnarray*}
    By setting $n=4,5$, we know $B\leq0$, thus
    \begin{eqnarray*}
        \kappa(v_{i},u_{0})=\cfrac{n-m-2}{n-1}.
    \end{eqnarray*}
    Moreover, by setting $n\geq6,m=1$, we know $B\leq0$, thus
    \begin{eqnarray*}
        \kappa(v_{i},u_{0})=\cfrac{n-m-2}{n-1}.
    \end{eqnarray*}
    Hence, we consider $\kappa(v_{i},u_{0})$ in the case of $n\geq6,2\leq m\leq n-3$.\par

        We consider a coupling between $m_{v_{i}}$ and $m_{u_{0}}$.  Our transfer plan moving $m_{v_{i}}$ to $m_{u_{0}}$ should be as follows:
     \begin{enumerate}
\renewcommand{\labelenumi}{(\arabic{enumi})}
\item Move the mass of $\cfrac{1}{n-1}$ from $z$ to $z$ for any $z\in\Gamma(v_{i})\cap\Gamma(u_{0})$. The distance is $0$.
\item Move the mass of $\cfrac{1}{(m+1)(n-1)}$ from any $x\in V\setminus\Gamma(u_{0})$ to $v_{i}$. The distance is $1$.
\item Move the mass of $\cfrac{m}{(m+1)(n-1)(n-m-1)}$ from any $x\in V\setminus\Gamma(u_{0})$ to $y\in\Gamma(u_{0})$. The distance is $1$.
\end{enumerate}

    By this transfer plan, we define a map $A:V\times V\to \mathbb{R}$ by
    \begin{eqnarray*}
  A(x,y)=
  \begin{cases}
    \cfrac{1}{n-1}, & \text{if $x=y\in\Gamma(v_{i})\cap\Gamma(u_{0})$}, \\
    \cfrac{1}{(m+1)(n-1)},                 & \text{if $x\in V\setminus\Gamma(u_{0})$, y=$v_{i}$}, \\
    \cfrac{m}{(m+1)(n-1)(n-m-1)} ,      & \text{if $x\in V\setminus\Gamma(u_{0})$, $y\in\Gamma(u_{0})$},\\
    0,    & \text{otherwise}.
  \end{cases}
\end{eqnarray*}

    By definition of $W(m_{v_{i}},m_{u_{0}})$, we get
    \begin{eqnarray*}
        W(m_{v_{i}},m_{u_{0}})&\leq&\cfrac{1}{(m+1)(n-1)}\times(m+1)\\
        &\ &+\cfrac{m}{(m+1)(n-1)(n-m-1)}\times(m+1)\times(n-m-1)=\cfrac{m+1}{n-1},        
    \end{eqnarray*}
    which implies
    \begin{eqnarray*}
        \kappa(v_{i},u_{0})\geq1-\cfrac{m+1}{n-1}=\cfrac{n-m-2}{n-1}.
    \end{eqnarray*}   

 Thus, we have
    \begin{eqnarray*}
        \kappa(v_{i},u_{0})=\cfrac{n-m-2}{n-1}=\cfrac{\#(v_{i},u_{0})}{d_{v_{i}}\vee d_{u_{0}}}.
    \end{eqnarray*}
    This completes the proof.
    
\end{proof}

\begin{corollary}
    We consider a graph $K_{n}-[(n-2)e]$ $(n\geq4)$. Then for any pair of neighboring vertices $x,y$, we have
    \begin{eqnarray*}
        \kappa(x,y)=\cfrac{\#(x,y)}{d_{x}\vee d_{y}}=
        \begin{cases}
    \cfrac{n-3}{n-1},       & \text{if $x=v_{1}$, $y=u_{i}$}, \\
    \cfrac{n-3}{n-2} ,      & \text{if $x=u_{i}$, $y=u_{j}$},\\
    0 , &\text{if $x=v_{1},y=u_{0}$}.
  \end{cases}
    \end{eqnarray*}
\end{corollary}

\subsection{Case where edges of a cycle are removed}
In this subsection, we consider the graph $K_{n}-C_{m}$ obtained by removing the edges of a cycle from a complete graph.
First, we relabeled the vertices according to their degree (see Figure \ref{サイクルの例}). Let $v_{i}\ (i=1,\ldots,n-m)$ denote the vertex of degree $n-1$, $u_{i}\ (i=1,\ldots,m)$ denote the vertex of degree $n-3$ on $K_{n}-C_{m}$. 

\begin{figure}[H]
\centering
\includegraphics[height=6cm]{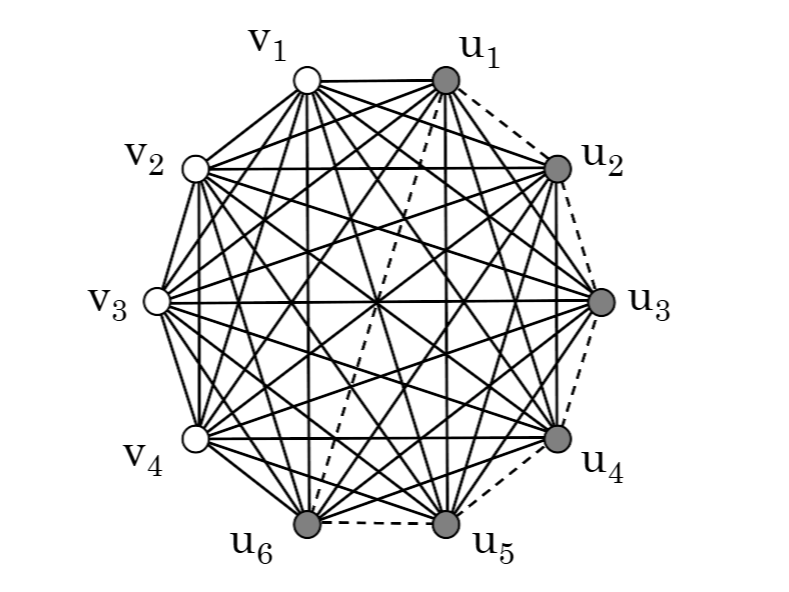}
\caption{$K_{10}-C_{6}$}
\label{サイクルの例}
\end{figure}

We provide the following lemma concerning the number of triangles in the case of a graph excluding cycle edges to calculate Ollivier Ricci curvature.

\begin{lemma}
\label{サイクルの三角形数}
    We consider a graph $K_{n}-C_{m}$, where $3\leq m\leq n-1$.
    Moreover, for some pair of neighboring vertices of degree $n-3$, denote $u_{i},u_{j}$ if $m\geq4$ and there is one or two vertices which are not adjacent to $u_{i}$ and $u_{j}$, and denote $u'_{i},u'_{j}$ if $m\geq6$ and there is no vertex which is not adjacent to $u'_{i}$ and $u'_{j}$.
    Then for any pair of neighboring two vertices, we have
    \begin{eqnarray*}
        \#(v_{i},v_{j})&=&n-2,\\
        \#(v_{i},u_{j})&=&n-4,\\
        \#(u_{i},u_{j})&=&
        \begin{cases}
    n-4, & \text{if $m=4$}, \\
    n-5,  & \text{if $m\geq5$ },
    \end{cases}\\
        \#(u'_{i},u'_{j})&=&n-6.
    \end{eqnarray*}
\end{lemma}

    \begin{proof}
         First, we consider any pair of neighboring vertices $v_{i},v_{j}$. We know that vertex $v_{i}$ has $n-2$ neighbors, except for $v_{j}$, and $v_{j}$ is adjacent to all of them. Therefore,
    \begin{eqnarray*}
        \#(v_{i},v_{j})=n-2.
    \end{eqnarray*}
    
    Second, we consider any pair of neighboring vertices $v_{i},u_{j}$. We know that $u_{j}$ has $n-4$ neighbors, except for $v_{i}$, and $v_{i}$ is adjacent to all of them. Therefore,
    \begin{eqnarray*}
        \#(v_{i},u_{j})=n-4.
    \end{eqnarray*}

    Third, we consider any pair of neighboring vertices $u_{i},u_{j}$. 
    \begin{itemize}
    \item $m=4$\\
    We know that there are two vertices that are not adjacent to $u_{i}$ and $u_{j}$. Therefore,
    \begin{eqnarray*}
        \#(u_{i},u_{j})=(n-2)-2=n-4.
    \end{eqnarray*}

    \item  $m\geq5$ and there is the vertex which is not adjacent to $u_{i}$ and $u_{j}$\\
    We know that each vertex has two non-adjacent vertices, one of which is the same vertex. Therefore,
    \begin{eqnarray*}
        \#(u_{i},u_{j})=(n-2)-3=n-5.
    \end{eqnarray*}
    \end{itemize}

    Finally, we consider any pair of neighboring vertices  $u'_{i}$ and $u'_{j}$\\
    We know that each vertex has two non-adjacent vertices, and all of them are different. Therefore,
    \begin{eqnarray*}
        \#(u'_{i},u'_{j})=(n-2)-4=n-6.
    \end{eqnarray*}
    
     \end{proof}

By applying Lemma\ref{サイクルの三角形数}, the curvature is calculated as:

\begin{proposition}
    We consider a graph $K_{n}-C_{m}$, where $3\leq m\leq n-1$. Moreover, for some pair of neighboring vertices of degree $n-3$, denote $u_{i},u_{j}$ if $m\geq4$ and there is one or two vertices which are not adjacent to $u_{i}$ and $u_{j}$, and denote $u'_{i},u'_{j}$ if $m\geq6$ and there is no vertex which is not adjacent to $u'_{i}$ and $u'_{j}$.
    Then for any pair of neighboring vertices $x,y$, we have
    \begin{eqnarray*}
        \kappa(x,y)=\cfrac{\#(x,y)}{d_{x}\vee d_{y}}=
        \begin{cases}
    \cfrac{n-2}{n-1}, & \text{if $x=v_{i},y=v_{j}$}, \\
    \cfrac{n-4}{n-1},       & \text{if $x=v_{i}$, $y=u_{j}$}, \\
    \cfrac{n-4}{n-3} ,  &\text{if $x=u_{i}$, $y=u_{j}$, and $m=4$},\\
    \cfrac{n-5}{n-3} ,      & \text{if $x=u_{i}$, $y=u_{j}$ and $m\geq5$},\\
    \cfrac{n-6}{n-3} , &\text{if $x=u'_{i},y=u'_{j}$}.
  \end{cases}
    \end{eqnarray*}
\end{proposition}

\begin{proof}
    First, we consider any pair of neighboring vertices $v_{i},v_{j}$. 
    By assumption and Lemma\ref{サイクルの三角形数}, we know 
    \begin{eqnarray*}
        d_{v_{i}}\vee d_{v_{j}}=d_{v_{i}}\wedge d_{v_{j}}=n-1, \#(v_{i},v_{j})=n-2.
    \end{eqnarray*}
    So, by Theorem \ref{リッチ曲率の下限} and Theorem \ref{リッチ曲率の上限},
    \begin{eqnarray*}
        \cfrac{n-2}{n-1}\geq
         \kappa(v_{i},v_{j})\geq -2\left( 1- \cfrac{1}{n-1} - \cfrac{1}{n-1} - \cfrac{n-2}{n-1} \right)_++ \cfrac{n-2}{n-1}= \cfrac{n-2}{n-1}.
    \end{eqnarray*}
    Thus,
    \begin{eqnarray*}
        \kappa(v_{i},v_{j})=\cfrac{n-2}{n-1}=\cfrac{\#(v_{i},v_{j})}{d_{v_{i}}\vee d_{v_{j}}}.
    \end{eqnarray*}

    Second, we consider any pair of neighboring vertices  $v_{i},u_{j}$. By assumption and Lemma\ref{サイクルの三角形数}, we know
    \begin{eqnarray*}
        d_{v_{i}}\vee d_{u_{j}}=n-1,d_{v_{i}}\wedge d_{u_{j}}=n-3,\#(v_{i},u_{j})=n-4.
    \end{eqnarray*}
    So, by Theorem \ref{リッチ曲率の下限} and Theorem \ref{リッチ曲率の上限},
    \begin{eqnarray*}
       \cfrac{n-4}{n-1}\geq
       \kappa(v_{i},u_{j})&\geq& -\left( 1- \cfrac{1}{n-1} - \cfrac{1}{n-3} - \cfrac{n-4}{n-3} \right)_+\\
       &\ & - \left( 1- \cfrac{1}{n-1} - \cfrac{1}{n-3} - \cfrac{n-4}{n-1} \right)_+ + \cfrac{n-4}{n-1}\\
       &=& -\left(\cfrac{n-5}{(n-1)(n-3)} \right)_+ +\cfrac{n-4}{n-1}. 
    \end{eqnarray*}
    Hence, we need to consider $\kappa(v_{i},u_{j})$ in the case of $n\geq6$.\par
    
    \begin{itemize}
    \item $m=3$\\
        We consider a coupling between $m_{v_{i}}$ and $m_{u_{j}}$. Our transfer plan for moving $m_{v_{i}}$ to $m_{u_{j}}$ should be as follows:
     \begin{enumerate}
\renewcommand{\labelenumi}{(\arabic{enumi})}
\item Move the mass of $\cfrac{1}{n-1}$ from $z$ to $z$ for any $z\in\Gamma(v_{i})\cap\Gamma(u_{j})$. The distance is $0$.
\item Move the mass of $\cfrac{1}{3(n-1)}$ from any $x\in V\setminus\Gamma(u_{j})$ to $v_{i}$. The distance is $1$.
\item Move the mass of $\cfrac{2}{3(n-1)(n-3)}$ from any $x\in V\setminus\Gamma(u_{j})$ to $y\in\Gamma(u_{j})$. The distance is $1$.
\end{enumerate}
     By this transfer plan, we define a map $A:V\times V\to \mathbb{R}$ by
    \begin{eqnarray*}
  A(x,y)=
  \begin{cases}
    \cfrac{1}{n-1}, & \text{if $x=y\in\Gamma(v_{i})\cap\Gamma(u_{j})$}, \\
    \cfrac{1}{3(n-1)},                 & \text{if $x\in V\setminus\Gamma(u_{j})$, y=$v_{i}$}, \\
    \cfrac{2}{3(n-1)(n-3)} ,      & \text{if $x\in V\setminus\Gamma(u_{j})$, $y\in\Gamma(u_{j})$},\\
    0,    & \text{otherwise}.
  \end{cases}
\end{eqnarray*}
    By the definition of $W(m_{v_{i}},m_{u_{j}})$, we obtain: 
    \begin{eqnarray*}
        W(m_{v_{i}},m_{u_{j}})&\leq&\cfrac{1}{3(n-1)}\times3 +\cfrac{2}{3(n-1)(n-3)}\times3\times(n-3)\\
        &=&\cfrac{3}{n-1},        
    \end{eqnarray*}
    which implies
    \begin{eqnarray*}
        \kappa(v_{i},u_{j})\geq1-\cfrac{3}{n-1}=\cfrac{n-4}{n-1}.
    \end{eqnarray*}    
    Thus, we have
    \begin{eqnarray*}
        \kappa(v_{i},u_{j})=\cfrac{n-4}{n-1}.
    \end{eqnarray*}

    \item $m=4$\\
    Let denote $\hat{u_{j}}$ that the vertex is adjacent to $u_{j}$.  
   We consider a coupling between $m_{v_{i}}$ and $m_{u_{j}}$. Our transfer plan for moving $m_{v_{i}}$ to $m_{u_{j}}$ should be as follows:
\begin{enumerate}
\renewcommand{\labelenumi}{(\arabic{enumi})}
\item Move the mass of $\cfrac{1}{n-1}$ from $z$ to $z$ for any $z\in\Gamma(v_{i})\cap\Gamma(u_{j})$. The distance is $0$.

\item Using the mass at $u_{j}$, we have Move the mass of $\cfrac{2}{(n-1)(n-3)}$ to $\hat{u_{j}}$, and move the mass of $\cfrac{n-5}{(n-1)(n-3)}$ to $v_{i}$ . The distance is $1$.\par
We remark that $ \cfrac{1}{n-1}-\cfrac{2}{(n-1)(n-3)}\geq0$ and $ \cfrac{1}{n-3}-\cfrac{n-5}{(n-1)(n-3)}\geq0.$

\item Using the mass at $x\in V\setminus(\Gamma(u_{j})\cup\lbrace u_{j}\rbrace)$. Move the mass of $\cfrac{1}{(n-1)(n-3)}$ to $y\in\Gamma(u_{j})\setminus\lbrace v_{i},\hat{u_{j}}\rbrace$, and move the mass of $\cfrac{2}{(n-1)(n-3)}$ to $v_{i}$. The distance is $1$.\par
We remark that $ \cfrac{2}{n-1}-\cfrac{2}{(n-1)(n-3)}\times(n-5)\geq0.$

\end{enumerate}
     By this transfer plan, we define a map $A:V\times V\to \mathbb{R}$ by
    \begin{eqnarray*}
  A(x,y)=
  \begin{cases}
    \cfrac{1}{n-1}, & \text{if $x=y\in\Gamma(v_{i})\cap\Gamma(u_{j})$}, \\
    \cfrac{2}{(n-1)(n-3)},                 & \text{if $x=u_{j}$, $y=\hat{u_{j}}$}, \\
    \cfrac{n-5}{(n-1)(n-3)},    &\text{if $x=u_{j}, y=v_{i}$,}\\
    \cfrac{1}{(n-1)(n-3)} ,      & \text{if $x\in  V\setminus(\Gamma(u_{j})\cup\lbrace u_{j}\rbrace)$, $y\in\Gamma(u_{j})\setminus\lbrace v_{i},\hat{u_{j}}\rbrace$},\\
    \cfrac{2}{(n-1)(n-3)},   &\text{if $x\in  V\setminus(\Gamma(u_{j})\cup\lbrace u_{j}\rbrace)$, $y=v_{i}$},\\
    0,    & \text{otherwise}.
  \end{cases}
\end{eqnarray*}
    By the definition of $W(m_{v_{i}},m_{u_{j}})$, we obtain: 
    \begin{eqnarray*}
        W(m_{v_{i}},m_{u_{j}})&\leq&\cfrac{2}{(n-1)(n-3)} +\cfrac{n-5}{(n-1)(n-3)}\\
        &\ &+\cfrac{1}{(n-1)(n-3)}\times2\times(n-5)+\cfrac{2}{(n-1)(n-3)}\times2=\cfrac{3}{n-1},        
    \end{eqnarray*}
    which implies
    \begin{eqnarray*}
        \kappa(v_{i},u_{j})\geq1-\cfrac{3}{n-1}=\cfrac{n-4}{n-1}.
    \end{eqnarray*}

    \item $m=5,n=6$\\
   We consider a coupling between $m_{v_{i}}$ and $m_{u_{j}}$. We define a function $\psi:V\setminus(\Gamma(u_{j})\cup\lbrace u_{j}\rbrace)\to\Gamma(v_{i})\cap\Gamma(u_{j})$ by for any $x\in V\setminus(\Gamma(u_{j})\cup\lbrace u_{j}\rbrace)$, $d(x,\psi(x))=1$. It is clear that $\psi$ is a bijection map.  Our transfer plan for moving $m_{v_{i}}$ to $m_{u_{j}}$ should be as follows:
    \begin{enumerate}
\renewcommand{\labelenumi}{(\arabic{enumi})}
\item Move the mass of $\cfrac{1}{5}$ from $z$ to $z$ for any $z\in\Gamma(v_{i})\cap\Gamma(u_{j})$. The distance is $0$.

\item Move the mass of $\cfrac{1}{15}$ from $u_{j}$ to $y\in\Gamma(u_{j})$. The distance is $1$.

\item Using the mass at $x\in\ V\setminus(\Gamma(u_{j})\cup\lbrace u_{j}\rbrace)$. Move the mass of $\cfrac{2}{15}$ to $v_{i}$ and move the mass of $\cfrac{1}{15}$ to  $y=\psi(x)$. The distance is $1$.
\end{enumerate}

     By this transfer plan, we define a map $A:V\times V\to \mathbb{R}$ by
    \begin{eqnarray*}
  A(x,y)=
  \begin{cases}
    \cfrac{1}{5}, & \text{if $x=y\in\Gamma(v_{i})\cap\Gamma(u_{j})$}, \\
    \cfrac{1}{15},                 & \text{if $x=u_{j}$, $y\in\Gamma(u_{j})$}, \\
    \cfrac{2}{15},    &\text{if $x\in V\setminus(\Gamma(u_{j})\cup\lbrace u_{j}\rbrace), y=v_{i}$,}\\
    \cfrac{1}{15} ,      & \text{if $x\in V\setminus(\Gamma(u_{j})\cup\lbrace u_{j}\rbrace)$, $y=\psi(x)$},\\
    0,    & \text{otherwise}.
  \end{cases}
\end{eqnarray*}

   By the definition of $W(m_{v_{i}},m_{u_{j}})$, we obtain: 
    \begin{eqnarray*}
        W(m_{v_{i}},m_{u_{j}})\leq\cfrac{1}{15}\times3+\cfrac{2}{15}\times2+\cfrac{1}{15}\times2=\cfrac{3}{5}=\cfrac{3}{n-1},        
    \end{eqnarray*}
    which implies
    \begin{eqnarray*}
        \kappa(v_{i},u_{j})\geq1-\cfrac{3}{n-1}=\cfrac{n-4}{n-1}.
    \end{eqnarray*}

    \item $m\geq5,n\geq7$\par
    Let denote $u_{k},u_{l}\in\Gamma(u_{j})$ that each vertex has only one non-adjacent vertex $u_{p}\in\Gamma(u_{j})\cup\lbrace u_{j}\rbrace$.  
    We consider a coupling between $m_{v_{i}}$ and $m_{u_{j}}$. Our transfer plan for moving $m_{v_{i}}$ to $m_{u_{j}}$ should be as follows:

\begin{enumerate}
\renewcommand{\labelenumi}{(\arabic{enumi})}
\item Move the mass of $\cfrac{1}{n-1}$ from $z$ to $z$ for any $z\in\Gamma(v_{i})\cap\Gamma(u_{j})$. The distance is $0$.

\item 
Using the mass at $u_{j}$. Move the mass of $\cfrac{2}{(n-1)(n-3)}$ to $u_{k},u_{l}$ and the mass of $\cfrac{n-7}{(n-1)(n-3)}$ to $v_{i}$. The distance is $1$.\par
We remark that $ \cfrac{1}{n-1}-\cfrac{2}{(n-1)(n-3)}\times2\geq0$. 

\item Using the mass at $x\in V\setminus(\Gamma(u_{j})\cup\lbrace u_{j}\rbrace)$. Move the mass of $\cfrac{1}{(n-1)(n-3)}$ to $y\in\Gamma(u_{j})\setminus\lbrace v_{i},u_{k},u_{l}\rbrace$, and move the mass of $\cfrac{3}{(n-1)(n-3)}$ to $v_{i}$. The distance is $1$.
We remark that $ \cfrac{2}{n-1}-\cfrac{2}{(n-1)(n-3)}\times(n-6)\geq0$.
\end{enumerate}

    By this transfer plan, we define a map $A:V\times V\to \mathbb{R}$ by
    \begin{eqnarray*}
  A(x,y)=
  \begin{cases}
    \cfrac{1}{n-1}, & \text{if $x=y\in\Gamma(v_{i})\cap\Gamma(u_{j})$}, \\
    \cfrac{2}{(n-1)(n-3)},                 & \text{if $x=u_{j}$, $y\in\lbrace u_{k},u_{l}\rbrace$}, \\
    \cfrac{n-7}{(n-1)(n-3)},    &\text{if $x=u_{j}, y=v_{i}$,}\\
    \cfrac{1}{(n-1)(n-3)} ,      & \text{if $x\in V\setminus(\Gamma(u_{j})\cup\lbrace u_{j}\rbrace)$, $y\in\Gamma(u_{j})\setminus\lbrace v_{i},u_{k},u_{l}\rbrace$},\\
    \cfrac{3}{(n-1)(n-3)},   &\text{if $x\in  V\setminus(\Gamma(u_{j})\cup\lbrace u_{j}\rbrace)$,$y=v_{i}$},\\
    0,    & \text{otherwise}.
  \end{cases}
\end{eqnarray*}

    By the definition of $W(m_{v_{i}},m_{u_{j}})$, we obtain: 
    \begin{eqnarray*}
        W(m_{v_{i}},m_{u_{j}})&\leq&\cfrac{2}{(n-1)(n-3)}\times2 +\cfrac{n-7}{(n-1)(n-3)}\\
        &\ &+\cfrac{1}{(n-1)(n-3)}\times2\times(n-6)+\cfrac{3}{(n-1)(n-3)}\times2=\cfrac{3}{n-1},        
    \end{eqnarray*}
    which implies
    \begin{eqnarray*}
        \kappa(v_{i},u_{j})\geq1-\cfrac{3}{n-1}=\cfrac{n-4}{n-1}.
    \end{eqnarray*} 

Therefore, for any pair of neighboring vertices $v_{i},u_{j}$,
\begin{eqnarray*}
    \kappa(v_{i},u_{j})=\cfrac{n-4}{n-1}=\cfrac{\#(v_{i},u_{j})}{d_{v_{i}}\vee d_{u_{j}}}.
\end{eqnarray*}
       \end{itemize}

   Third, we consider any pair of neighboring vertices $u_{i},u_{j}$.

    \begin{itemize}
    \item $m=4$\\
 By assumption and Lemma\ref{サイクルの三角形数}, we know
    \begin{eqnarray*}
        d_{u_{i}}\vee d_{u_{j}}=d_{u_{i}}\wedge d_{u_{j}}=n-3,\#(u_{i},u_{j})=n-4.
    \end{eqnarray*}
        So, by Theorem \ref{リッチ曲率の下限} and Theorem \ref{リッチ曲率の上限},
    \begin{eqnarray*}
       \cfrac{n-4}{n-3}\geq
       \kappa(u_{i},u_{j})\geq -2\left( 1- \cfrac{1}{n-3} - \cfrac{1}{n-3} - \cfrac{n-4}{n-3} \right)_++ \cfrac{n-4}{n-3}= \cfrac{n-4}{n-3}. 
    \end{eqnarray*}
    Thus,
        \begin{eqnarray*}
        \kappa(u_{i},u_{j})=\cfrac{n-4}{n-3}=\cfrac{\#(u_{i},u_{j})}{d_{u_{i}}\vee d_{u_{j}}}.
    \end{eqnarray*}
    
    \item $m\geq5$\\
  By assumption and Lemma\ref{サイクルの三角形数}, we know
    \begin{eqnarray*}
        d_{u_{i}}\vee d_{u_{j}}=d_{u_{i}}\wedge d_{u_{j}}=n-3,\#(u_{i},u_{j})=n-5.
    \end{eqnarray*}   
     So, by Theorem \ref{リッチ曲率の下限} and Theorem \ref{リッチ曲率の上限},
    \begin{eqnarray*}
       \cfrac{n-5}{n-3}\geq
       \kappa(u_{i},u_{j})\geq -2\left( 1- \cfrac{1}{n-3} - \cfrac{1}{n-3} - \cfrac{n-5}{n-3} \right)_++ \cfrac{n-5}{n-3}= \cfrac{n-5}{n-3}. 
    \end{eqnarray*}
    Thus,
    \begin{eqnarray*}
        \kappa(u_{i},u_{j})=\cfrac{n-5}{n-3}=\cfrac{\#(u_{i},u_{j})}{d_{u_{i}}\vee d_{u_{j}}}.
    \end{eqnarray*}
    
    \end{itemize}

Finally, we consider any pair of neighboring vertices $u'_{i},u'_{j}$.
  By assumption and Lemma\ref{サイクルの三角形数}, we know
    \begin{eqnarray*}
        d_{u'_{i}}\vee d_{u'_{j}}=d_{u'_{i}}\wedge d_{u'_{j}}=n-3,\#(u'_{i},u'_{j})=n-6.
    \end{eqnarray*}   
     So, by Theorem \ref{リッチ曲率の下限} and Theorem \ref{リッチ曲率の上限},
    \begin{eqnarray*}
       \cfrac{n-6}{n-3}\geq
       \kappa(u'_{i},u'_{j})\geq -2\left( 1- \cfrac{1}{n-3} - \cfrac{1}{n-3} - \cfrac{n-6}{n-3} \right)_++ \cfrac{n-6}{n-3}= -2\left( \cfrac{1}{n-3} \right)_++ \cfrac{n-6}{n-3}. 
    \end{eqnarray*}
    Hence, we need to consider $\kappa(u'_{i},u'_{j})$ in the case of $m\geq6$.\par

\begin{itemize}
    \item $m=6$\\
    We consider a coupling between $m_{u'_{i}}$ and $m_{u'_{j}}$.
    We define a function $\phi:\Gamma(u'_{i})\setminus(\Gamma(u'_{j})\cup\lbrace u'_{j}\rbrace)\to\Gamma(u'_{j})\setminus(\Gamma(u'_{i})\cup\lbrace u'_{i}\rbrace)$ by for any $x\in\Gamma(u'_{i})\setminus(\Gamma(u'_{j})\cup\lbrace u'_{j}\rbrace)$, $d(x,\phi(x))=1$. It is clear that $\phi$ is a bijection map.   
    Our transfer plan for moving $m_{u'_{i}}$ to $m_{u'_{j}}$ should be as follows:
\begin{enumerate}
\renewcommand{\labelenumi}{(\arabic{enumi})}
\item Move the mass of $\cfrac{1}{n-3}$ from $z$ to $z$ for any $z\in\Gamma(u'_{i})\cap\Gamma(u'_{j})$. The distance is $0$.

\item Move the mass of $\cfrac{1}{3(n-3)}$ from $u'_{j}$ to $y\in\Gamma(u'_{j})\setminus\Gamma(u'_{i})$. The distance is $1$.

\item Using the mass at $x\in\Gamma(u'_{i})\setminus(\Gamma(u'_{j})\cup\lbrace u'_{j}\rbrace)$. Move the mass of $\cfrac{1}{3(n-3)}$ to $u'_{i}$, and move the mass of $\cfrac{2}{3(n-3)}$ to  $y=\phi(x)$. The distance is $1$.
\end{enumerate}
         By this transfer plan, we define a map $A:V\times V\to \mathbb{R}$ by
    \begin{eqnarray*}
  A(x,y)=
  \begin{cases}
    \cfrac{1}{n-3}, & \text{if $x=y\in\Gamma(u'_{i})\cap\Gamma(u'_{j})$}, \\
    \cfrac{1}{3(n-3)},                 & \text{if $x=u'_{j}$, $y\in\Gamma(u'_{j})\setminus\Gamma(u'_{i})$}, \\
    \cfrac{1}{3(n-3)},    &\text{if $x\in\Gamma(u'_{i})\setminus(\Gamma(u'_{j})\cup\lbrace u'_{j}\rbrace), y=u'_{i}$,}\\
    \cfrac{2}{3(n-3)} ,      & \text{if $x\in\Gamma(u'_{i})\setminus(\Gamma(u'_{j})\cup\lbrace u'_{j}\rbrace)$, $y=\phi(x)$},\\
    0,    & \text{otherwise}.
  \end{cases}
\end{eqnarray*}

   By the definition of $W(m_{u'_{i}},m_{u'_{j}})$, we obtain: 
    \begin{eqnarray*}
        W(m_{u'_{i}},m_{u'_{j}})\leq\cfrac{1}{3(n-3)}\times3+\cfrac{1}{3(n-3)}\times2+\cfrac{2}{3(n-3)}\times2=\cfrac{3}{n-3},        
    \end{eqnarray*}
    which implies
    \begin{eqnarray*}
        \kappa(u'_{i},u'_{j})\geq1-\cfrac{3}{n-3}=\cfrac{n-6}{n-3}.
    \end{eqnarray*}   

    \item $m\geq7$, and of all the non-adjacent vertices of $u'_{i}$ or $u'_{j}$, only two are non-adjacent\\

    Let $\hat{u_{i}}\in\Gamma(u'_{i})$ denote the vertex described above.    
    We consider a coupling between $m_{u'_{i}}$ and $m_{u'_{j}}$. Our transfer plan for moving $m_{u'_{i}}$ to $m_{u'_{j}}$ should be as follows:
\begin{enumerate}
\renewcommand{\labelenumi}{(\arabic{enumi})}
\item Move the mass of $\cfrac{1}{n-3}$ from $z$ to $z$ for any $z\in\Gamma(u'_{i})\cap\Gamma(u'_{j})$. The distance is $0$.

\item Move the mass of $\cfrac{1}{n-3}$ from $\hat{u_{i}}$ to $u'_{i}$. The distance is $1$.

\item Move the mass of $\cfrac{1}{2(n-3)}$ from $x\in\Gamma(u'_{i})\setminus(\Gamma(u'_{j})\cup\lbrace \hat{u_{i}}\rbrace)$ to $y\in\Gamma(u'_{j})\setminus(\Gamma(u'_{i})\cup\lbrace u'_{i}\rbrace)$. The distance is $1$.
\end{enumerate}

By this transfer plan, we define a map $A:V\times V\to \mathbb{R}$ by
   \begin{eqnarray*}
  A(x,y)=
  \begin{cases}
    \cfrac{1}{n-3}, & \text{if $x=y\in\Gamma(u'_{i})\cap\Gamma(u'_{j})$}, \\
    \cfrac{1}{n-3},                 & \text{if $x=\hat{u_{i}}$, $y=u'_{i}$}, \\
    \cfrac{1}{2(n-3)},    &\text{if $x\in\Gamma(u'_{i})\setminus(\Gamma(u'_{j})\cup\lbrace \hat{u_{i}}\rbrace), y\in\Gamma(u'_{j})\setminus(\Gamma(u'_{i})\cup\lbrace u'_{i}\rbrace)$,}\\
    0,    & \text{otherwise}.
  \end{cases}
\end{eqnarray*}

   By the definition of $W(m_{u'_{i}},m_{u'_{j}})$, we obtain: 
    \begin{eqnarray*}
        W(m_{u'_{i}},m_{u'_{j}})\leq\cfrac{1}{n-3}+\cfrac{1}{2(n-3)}\times2\times2=\cfrac{3}{n-3},        
    \end{eqnarray*}
    which implies
    \begin{eqnarray*}
        \kappa(u'_{i},u'_{j})\geq1-\cfrac{3}{n-3}=\cfrac{n-6}{n-3}.
    \end{eqnarray*}    

    \item $m\geq8$, and all the non-adjacent vertices of $u'_{i}$ or $u'_{j}$ are adjacent to each other\\
    We consider a coupling between $m_{u'_{i}}$ and $m_{u'_{j}}$. Our transfer plan for moving $m_{u'_{i}}$ to $m_{u'_{j}}$ should be as follows:

    \begin{enumerate}
\renewcommand{\labelenumi}{(\arabic{enumi})}
\item Move the mass of $\cfrac{1}{n-3}$ from $z$ to $z$ for any $z\in\Gamma(u'_{i})\cap\Gamma(u'_{j})$. The distance is $0$.

\item Move the mass of $\cfrac{1}{3(n-3)}$ from $x\in\Gamma(u'_{i})\setminus\Gamma(u'_{j})$ to $y\in\Gamma(u'_{j})\setminus\Gamma(u'_{i})$. The distance is $1$.
\end{enumerate}

  By this transfer plan, we define a map $A:V\times V\to \mathbb{R}$ by
   \begin{eqnarray*}
  A(x,y)=
  \begin{cases}
    \cfrac{1}{n-3}, & \text{if $x=y\in\Gamma(u'_{i})\cap\Gamma(u'_{j})$}, \\
    \cfrac{1}{3(n-3)},    &\text{if $x\in\Gamma(u'_{i})\setminus\Gamma(u'_{j}), y\in\Gamma(u'_{j})\setminus\Gamma(u'_{i})$,}\\
    0,    & \text{otherwise}.
  \end{cases}
\end{eqnarray*}

   By the definition of $W(m_{u'_{i}},m_{u'_{j}})$, we obtain: 
    \begin{eqnarray*}
        W(m_{u'_{i}},m_{u'_{j}})\leq\cfrac{1}{3(n-3)}\times3\times3=\cfrac{3}{n-3},        
    \end{eqnarray*}
    which implies
    \begin{eqnarray*}
        \kappa(u'_{i},u'_{j})\geq1-\cfrac{3}{n-3}=\cfrac{n-6}{n-3}.
    \end{eqnarray*} 

    Therefore, for any pair of neighboring vertices $u'_{i},u'_{j}$,
    \begin{eqnarray*}
        \kappa(u'_{i},u'_{j})=\cfrac{n-6}{n-3}=\frac{\#(u'_{i},u'_{j})}{d_{u'_{i}}\vee d_{u'_{j}}}.
    \end{eqnarray*}

\end{itemize}

This completes the proof.
\end{proof}

\begin{figure}[H]
\centering
\begin{minipage}[b]{0.45\columnwidth}
    \centering
    \includegraphics[width=1.0\columnwidth]{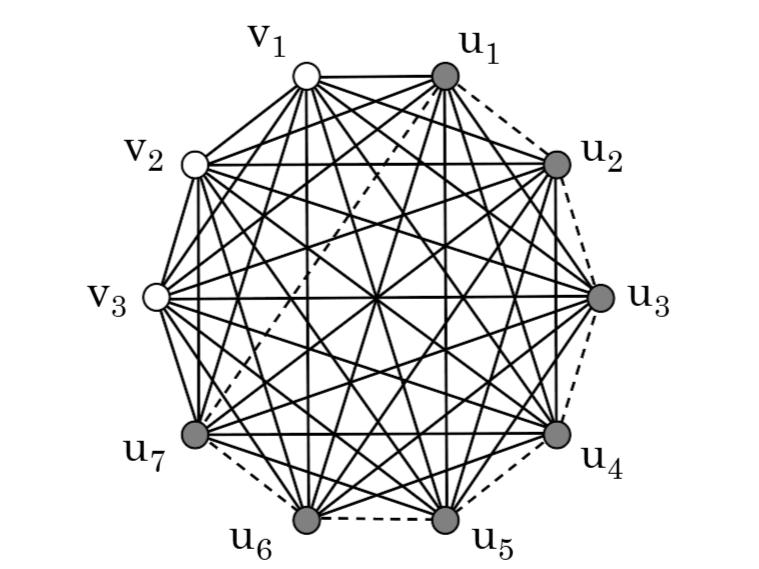}
    \caption{$K_{10}-C_{7}$}
    \label{サイクルの例１}
\end{minipage}
\begin{minipage}[b]{0.45\columnwidth}
    \centering
    \includegraphics[width=1.0\columnwidth]{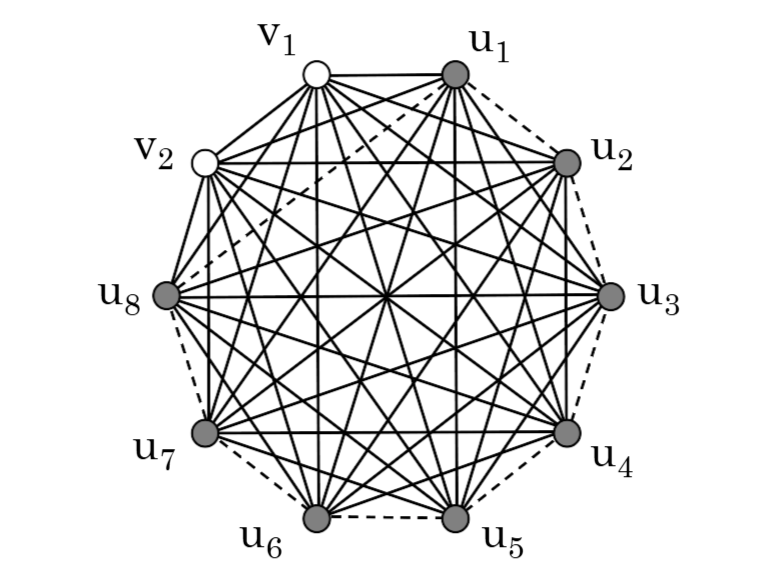}
    \caption{$K_{10}-C_{8}$}
    \label{サイクルの例２}
\end{minipage}
\end{figure}

   We consider the graphs in Figure\ref{サイクルの例},\ref{サイクルの例１},\ref{サイクルの例２}. A pair of $u_{1}$ and $u_{3}$ has a vertex $u_{2}$ which is not adjacent to the two vertices included in the pair, so that the pair is denoted as $u_{1},u_{3}$. A pair of $u_{1}$ and $u_{4}$ has no vertex that is not adjacent to two vertices included in the pair; thus, the pair is denoted as $u'_{1},u'_{4}$. Also, we consider the graph in Figure\ref{サイクルの例２}. For a pair of $u'_{1},u'_{4}$, two of all vertices the are non-adjacent to those two vertices are not adjacent. For a pair of $u'_{1},u'_{5}$, all vertices that are non-adjacent to these two vertices are adjacent.

\begin{corollary}
    We consider a graph $K_{n}-C_{n},(n\geq6)$.
    For some pairs of neighboring vertices, denote $u_{i},u_{j}$ if there is one vertex that is not adjacent to $u_{i}$ and $u_{j}$, and denote $u'_{i},u'_{j}$ if there is no vertex that is not adjacent to $u'_{i}$ and $u'_{j}$.
    Then for any neighboring $x,y$, we have
     \begin{eqnarray*}
        \kappa(x,y)=\cfrac{\#(x,y)}{d_{x}\vee d_{y}}=
        \begin{cases}
    \cfrac{n-5}{n-3} ,      & \text{if $x=u_{i}$, $y=u_{j}$ },\\
    \cfrac{n-6}{n-3} , &\text{if $x=u'_{i},y=u'_{j}$}.
  \end{cases}
    \end{eqnarray*} 
\end{corollary}

\section{Conclusion}
\label{Conclusion}
In this study, we investigated how the Ollivier–Ricci curvature changes when edges are removed from a complete graph, which is known to have the highest curvature, in various patterns.
We found that the Ollivier–Ricci curvature remains positive not only when edges are removed uniformly from each vertex of the complete graph (as in the case of matching edges) but also when edges are removed from specific vertices (as in the case of cycles). 
Furthermore, we prove that the Ricci curvature remains non-negative even when edges are removed intensively from a single vertex. 
This confirms that a complete graph possesses an extremely robust structure. Furthermore, to consider the conditions for positive or negative curvature, we examine the necessity of removing more edges from complete graphs or devising methods for removing edges.

\section*{Funding}
The second named author was supported in part by JSPS KAKENHI [grant number 25K17253].

\section*{Conflict of interest}
None.

\bibliography{bibliography}
\bibliographystyle{plain}

\address{
Interdisciplinary Faculty of Science and Engineering\\
Shimane University\\
Matsue 690-8504 \\
Japan
}
{n25m001@matsu.shimane-u.ac.jp}

\address{%
Interdisciplinary Faculty of Science and Engineering\\
Shimane University\\
Matsue 690-8504 \\
Japan
}
{taiki\_yamada@riko.shimane-u.ac.jp}

\end{document}